\documentclass{amsart}
\usepackage{blindtext}
\usepackage{graphicx}
\usepackage{tikz}
\usepackage[toc,page]{appendix} 
\usepackage[utf8]{inputenc}
\usepackage{amssymb,amsfonts,amsmath}
\usepackage{enumerate}
\usepackage{mathrsfs}
\usepackage{eucal}
\usepackage[letterpaper, portrait,margin=1.7in]{geometry}
\usepackage{amsthm}
\usepackage[english]{babel}
\usepackage{comment}
\usepackage[all,cmtip,2cell]{xy}
\usepackage{lipsum}
\usepackage[makeroom]{cancel}
\usepackage{lmodern}
\usepackage{mathtools}
\usepackage[final]{pdfpages}
\usepackage{fancyhdr}
\DeclarePairedDelimiter\ceil{\lceil}{\rceil}
\DeclarePairedDelimiter\floor{\lfloor}{\rfloor}
\newcommand{\Mod}[1]{\ (\mathrm{mod}\ #1)}

\begin{document}

\theoremstyle{plain}

\newtheorem{thm}{Theorem}[section]
\newtheorem{prop}[thm]{Proposition}
\newtheorem{lem}[thm]{Lemma}

\theoremstyle{definition}

\newtheorem{cor}[thm]{Corollary}

\newtheorem{term}[thm]{Terminology}
\newtheorem{defn}[thm]{Definition}
\newtheorem{defns}[thm]{Definitions}
\newtheorem{con}[thm]{Construction}
\newtheorem{exmp}[thm]{Example}
\newtheorem{nexmp}[thm]{Non-Example}
\newtheorem{exmps}[thm]{Examples}
\newtheorem{notn}[thm]{Notation}
\newtheorem{notns}[thm]{Notations}
\newtheorem{addm}[thm]{Addendum}
\newtheorem{exer}[thm]{Exercise}
\newtheorem{obs}[thm]{Observation}

\theoremstyle{remark}
\newtheorem{rem}[thm]{Remark}
\newtheorem{claim}[thm]{Claim}

\newtheorem{rems}[thm]{Remarks}
\newtheorem{warn}[thm]{Warning}
\newtheorem{sch}[thm]{Scholium}

\theoremstyle{plain} 
\newcommand{\thistheoremname}{}
\newtheorem{genericthm}[thm]{\thistheoremname}
\newenvironment{namedthm}[1]
  {\renewcommand{\thistheoremname}{#1}%
   \begin{genericthm}}
  {\end{genericthm}}

\begin{titlepage}
\title{Braid Groups on Triangulated Surfaces and Singular Homology}
\author{Karthik Yegnesh}
\end{titlepage}
\pagestyle{fancy} 
\fancyhf{} 
\fancyfoot[R]{\thepage}

\fancypagestyle{plain}{%
    \renewcommand{\headrulewidth}{0pt}%
    \fancyhf{}%
    \fancyfoot[R]{\thepage}%
}


\vspace{-20mm}

\maketitle
\begin{abstract}
Let $\Sigma_g$ denote the closed orientable surface of genus $g$ and fix an arbitrary simplicial triangulation of $\Sigma_g$. We construct and study a natural surjective group homomorphism from the surface braid group on $n$ strands on $\Sigma_g$ to the first singular homology group of $\Sigma_g$ with integral coefficients. In particular, we show that the kernel of this homomorphism is generated by canonical braids which arise from the triangulation of $\Sigma_g$. This provides a simple description of natural subgroups of surface braid groups which are closely tied to the homology groups of the surfaces $\Sigma_g$. 
\end{abstract}
\section{Introduction and Background}

The classical braid groups $B_n$ on the plane were introduced by Artin \cite{name}. Geometrically, elements of such braid groups appear as a collection of $n$ paths emanating from a set of $n$ distinct points on the plane which wind around each other and return to some permutation of the original set of points. Braid groups play an important role in various areas of mathematics, including the knot theory, representation theory, and the study of monodromy invariants in algebraic geometry. It is a well-known result of Artin \cite{name} that the group $B_n$ has the presentation
\begin{equation}
    B_n\simeq\langle\sigma_1, \sigma_2, \ldots, \sigma_{n-1}|\sigma_i\sigma_{i+1}\sigma_i=\sigma_{i+1}\sigma_i\sigma_{i+1}, \sigma_i\sigma_j=\sigma_j\sigma_i\rangle
\end{equation}
where $1\leq i\leq n-2$ in the first group of relations and $|i-j|\geq 2$ in the second group of relations. The generators $\sigma_i$ correspond to ``transposition" braids which swap adjacent points.\\
\indent Zariski \cite{zar} later provided a natural generalization of these notions by considering braid groups on more general surfaces. Let us restrict our attention to $\Sigma_g$, the closed orientable surface of genus $g$. Let $\Sigma_g^n$ denote the $n$-fold cartesian product of $\Sigma_g$ with itself and let $F_n(\Sigma_g)$ denote the \textit{$n^{th}$ ordered configuration space} of $\Sigma_g$, i.e the space 
$$
F_n(\Sigma_g)=\{(x_1, \ldots, x_n)\in\Sigma_g^n|x_i\neq x_j, \forall i\neq j\}.
$$
\indent Note that the symmetric group $S_n$ acts freely on $F_n(\Sigma_g)$ by permuting coordinates. We define the  \textit{$n^{th}$ configuration space} $C_n(\Sigma_g)$ of $\Sigma_g$ as the orbit space $C_n(\Sigma_g)=F_n(\Sigma_g)/S_n$. Note that $C_n(\Sigma_g)$ is a 2-manifold since $F_n(\Sigma_g)\subset\Sigma_g^n$ is an open subset and the permutation action of $S_n$ on $F_n(\Sigma_g)$ is free. We define the \textit{$n^{th}$ braid group} $B_n(\Sigma_g)$ of $\Sigma_g$ as the fundamental group$$B_n(\Sigma_g)=\pi_1 (C_n(\Sigma_g), [x_1, \ldots, x_n])$$

where $[x_1, \ldots, x_n]\in C_n(\Sigma_g)$ is a set of unordered points in $F_n(\Sigma_g)$. Since we are working with connected surfaces, we usually leave the basepoint implicit in our notation.
\begin{exmp}
The classical braid group $B_n$ is the fundamental group of the $n$-fold configuration space $C_n(\mathbb{R}^2)$ of the plane. In fact, $C_n(\mathbb{R}^2)$ is an Eilenberg-Maclane space $K(B_n,1)$.
\end{exmp}

\begin{rem}
If $n=1$, then $F_n(\Sigma_g)=\Sigma_g$, so $B_n(\Sigma_g)=\pi_1(\Sigma_g)$. Thus, braid groups on surfaces can be viewed as generalizations of their fundamental groups.
\end{rem}
Every element $[\sigma]\in B_n(\Sigma_g)$ induces a permutation of the elements of its basepoint. Thus, we obtain a surjective group homomorphism $f:B_n(\Sigma_g)\rightarrow S_n$ (which implies that $B_n(\Sigma_g)$ is nonabelian for $n\geq 3$). The subgroup $\mathrm{ker}(f)\subset B_n(\Sigma_g)$ is called the \textit{pure braid group} on $\Sigma_g$ and is denoted $P_n(\Sigma_g)$. By definition, $P_n(\Sigma_g)$ is a normal subgroup of index $n!$. Note also that $P_n(\Sigma_g)\simeq\pi_1 F_n(\Sigma_g)$. These groups fit into a canonical short exact sequence$$
1\longrightarrow P_n(\Sigma_g)\longrightarrow B_n(\Sigma_g)\longrightarrow S_n\longrightarrow 1.
$$

 Let $p^i:F_n(\Sigma_g)\rightarrow\Sigma_g$ denote the projection onto the $i^{th}$ coordinate for $1\leq i\leq n$. Then we have induced maps $$p^i_*:P_n(\Sigma_g,(x_1, \ldots, x_n))\longrightarrow\pi_1(\Sigma_g, x_i)
$$

for each $1\leq i\leq n$. The \textit{vertex loop} of $x_i$ induced by $[\sigma]\in P_n(\Sigma_g)$ is $p^i_*([\sigma])$. We denote this by $[\sigma]_{x_i}$.

We study interactions between surface braid groups $B_n(\Sigma_g)$ and the singular homology groups of $\Sigma_g$. In particular, we study a natural group homomorphism$$
\omega:B_n(\Sigma_g)\longrightarrow H_1(\Sigma_g; \mathbb{Z})
$$

which maps a braid on $\Sigma_g$ to the integral homology class of the formal sum of the individual paths (viewed as singular 1-simplices) it induces on each element of the basepoint. We show that $\mathrm{ker}(\omega)$ is generated by simple braids which arise from triangulations of $\Sigma_g$. Generally, $B_n(\Sigma_g)$ is a complicated object for arbitrary $g$, while the homology groups $H_1(\Sigma_g; \mathbb{Z})\simeq\mathbb{Z}^{2g}$ are well-understood. Thus, our results describe a useful and well-behaved relationship between braid groups on surfaces and homology groups. 
\section{Acknowledgments} 

I'd like to thank my mentor Gus Lonergan for his guidance in this project and Prof. Roman Bezrukavnikov for suggesting the project. I would also like to thank Dr. John Rickert and Dr. Tanya Khovanova for their advice on mathematical writing. Additionally, I thank Daniel Vitek for numerous helpful comments and revisions. Also, I thank the Research Science Institute, the Center for Excellence in Education, and the Massachusetts Institute for Technology for supporting this research.

\section{Preliminary Constructions}

\subsection{Construction of the Homomorphism $\omega$}\label{pre}
Let $B_n(\Sigma_g)$\footnote{$\omega$ can be constructed in the same way for braid groups on general topological spaces.} be based at $[x_1, \ldots x_n]\in C_n(\Sigma_g)$. We fix this basepoint throughout the paper. Let $[\phi]\in B_n(\Sigma_g)$ be a braid and let $\pi:F_n(\Sigma_g)\rightarrow C_n(\Sigma_g)$ denote the quotient map. Let $(x_1, \ldots, x_n)$ be an element in the fiber of $[x_1, \ldots, x_n]$ under $\pi$. The map $\pi$ is an $S_n$-cover, so we can lift $\phi$ to a unique path $\tilde{\phi}:[0,1]\rightarrow F_n(\Sigma_g)$ with $\tilde{\phi}(0)=(x_1, \ldots, x_n)$. 
\begin{prop}\label{zero}
Let $[\phi]\in B_n(\Sigma_g)$. Then $\sum_{i=1}^n(p_i\circ\tilde{\phi})\in\mathrm{ker}(\partial_1)$, where $\partial_1$ denotes the boundary operator.
\end{prop}
\begin{proof} 
We have that 
\begin{equation}
\partial_1(\sum_{i=1}^n(p_i\circ\tilde{\phi}))=\sum_{i=1}^n\partial_1(p_i\circ\tilde{\phi})
\end{equation}
\begin{equation}
=\sum_{i=1}^n(p_i\circ\tilde{\phi})|_{\{1\}}-\sum_{i=1}^n(p_i\circ\tilde{\phi})|_{\{0\}}
\end{equation}
By definition of $C_n(\Sigma_g)$, the $n$-tuple $(p_1\circ\tilde{\phi}|_{\{1\}}, \ldots, p_n\circ\tilde{\phi}|_{\{1\}})$ is a permutation of $(p_1\circ\tilde{\phi}|_{\{0\}}, \ldots,p_n\circ\tilde{\phi}|_{\{0\}})$, so $\sum_{i=1}^n(p_i\circ\tilde{\phi})|_{\{1\}}=\sum_{i=1}^n(p_i\circ\tilde{\phi})|_{\{0\}}$ and thus $\partial_1(\sum_{i=1}^n(p_i\circ\tilde{\phi}))$ vanishes.
\end{proof}

Thus, $\sum_{i=1}^n(p_i\circ\tilde{\phi})$ is a singular 1-cycle and hence yields a homology class, so that we can define a function $\omega:B_n(\Sigma_g)
\rightarrow H_1(\Sigma_g; \mathbb{Z})$ sending $[\phi]\mapsto [\sum_{i=1}^n(p_i\circ\tilde{\phi})]$.

\begin{prop}
The function $\omega$ is a well-defined group homomorphism.
\end{prop}
\begin{proof}
That $\omega$ is a group homomorphism follows from a straightforward computation, so we'll just show that it is well-defined. This is a consequence of the fact that $\pi$ is a covering map. Explicitly, suppose that $\phi_1,\phi_2:[0,1]\rightarrow C_n(\Sigma_g)$ are homotopic loops via $h:[0,1]\times[0,1]\rightarrow C_n(\Sigma_g)$ rel $\{0,1\}$. Fix unique lifts $\tilde{\phi}_1, \tilde{\phi}_2:[0,1]\rightarrow F_n(\Sigma_g)$ with $\tilde{\phi}_1(0)=\tilde{\phi}_2(0)=(x_1, \ldots, x_n)$. We need to show that $\sum_{i=1}^n(p_i\circ\tilde{\phi}_1-p_i\circ\tilde{\phi}_2)\in\mathrm{im}(\partial_2)$. By the Homotopy Lifting Property, we can lift $h$ to a homotopy $\tilde{h}:[0,1]\times[0,1]\rightarrow F_n(\Sigma_g)$ rel $\{0,1\}$ between $\tilde{\phi}_1$ and $\tilde{\phi}_2$. It follows that $p_i\circ\tilde{\phi}_1$ and $p_i\circ\tilde{\phi}_2$ are homotopic, so that $p_i\circ\tilde{\phi}_1-p_i\circ\tilde{\phi}_2\in\mathrm{im}(\partial_2)$, whence the result.
\end{proof}
We call $\omega$ the \textit{total winding number map} and refer to elements $[\sigma]\in\mathrm{ker}(\omega)$ as \textit{balanced braids} (See Figure \ref{balance}). Elements in $P_n(\Sigma_g)\cap\mathrm{ker}(\omega)$ are referred to as \textit{pure balanced braids}. So, balanced braids are those braids whose individual strands ``wind" around with orientations that cancel.

\begin{figure}\label{balance}
 \centering
\begin{tikzpicture}

\draw[gray, ultra thick] (0,0) -- (3,0) -- (3,3) -- cycle;
\draw[gray, ultra thick] (0,0) -- (0,3) -- (3,3) -- cycle;
\draw[gray, ultra thick] (3,0) -- (6,0) -- (6,3) -- cycle;
\draw[gray, ultra thick] (3,0) -- (3,3) -- (6,3) -- cycle;
\draw[gray, ultra thick] (3,0) -- (0,-3) -- (3,-3) -- cycle;
\draw[gray, ultra thick] (0,0) -- (0,-3) -- (3,0) -- cycle;
\draw[gray, ultra thick] (3,0) -- (3,-3) -- (6,0) -- cycle;
\draw[gray, ultra thick] (6,0) -- (3,-3) -- (6,-3) -- cycle;
\draw[gray, ultra thick] (6,0) -- (9,0) -- (9,3) -- cycle;
\draw[gray, ultra thick] (6,0) -- (6,3) -- (9,3) -- cycle;
\draw[gray, ultra thick] (6,0) -- (6,-3) -- (9,0) -- cycle;
\draw[gray, ultra thick] (9,0) -- (6,-3) -- (9,-3) -- cycle;

\draw[gray, ultra thick] (0,-3) -- (3,-3) -- (0,-6) -- cycle;
\draw[gray, ultra thick] (3,-3) -- (0,-6) -- (3,-6) -- cycle;
\draw[gray, ultra thick] (3,-3) -- (3,-6) -- (6,-3) -- cycle;

\draw[gray, ultra thick] (6,-3) -- (6,-6) -- (3,-6) -- cycle;
\draw[gray, ultra thick] (6,-3) -- (9,-3) -- (6,-6) -- cycle;
\draw[gray, ultra thick] (9,-3) -- (9,-6) -- (6,-6) -- cycle;

\draw[->] (0,0) .. controls (3,2.5) .. (5.8, 2.9);
\draw[->] (6,3) .. controls (5,1) .. (3.2, 0.1);
\draw[->] (3,0) .. controls (1.5,-1) .. (0.1, -0.1);

\end{tikzpicture}
\caption{Balanced Braid on a Triangulation of the Torus $\Sigma_1$ (Viewed on its fundamental polygon)}
\end{figure}
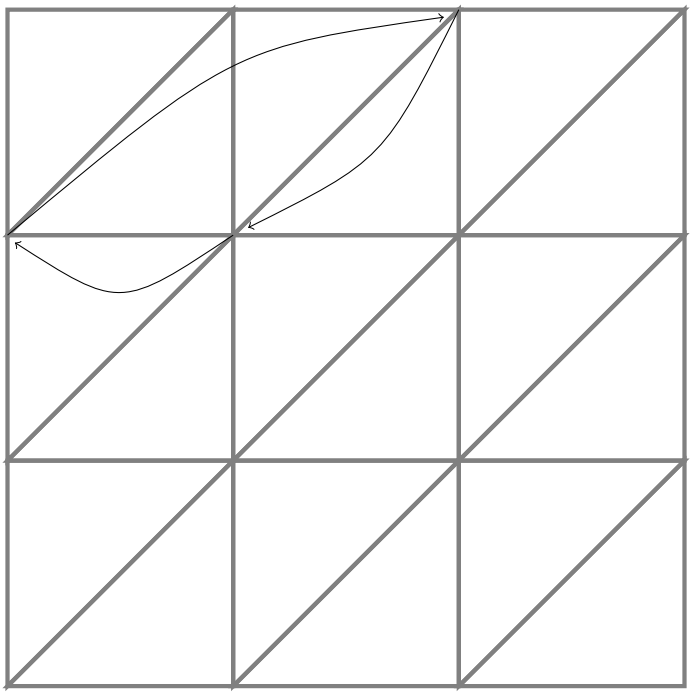

Let $j\in\{1, \ldots, n\}$ and let $B_j(\Sigma_g)$ be based at $[x_1, \ldots, x_j]\in C_n(\Sigma_g)$. For each such $j$, there is a natural group homomorphism
$$B_j(\Sigma_g)\longrightarrow B_n(\Sigma_g)$$
given by ``adding constant strands" to the points $x_{j+1}, \ldots, x_n$. In particular, for $j=1$, we obtain a homomorphism 
$$\upsilon:\pi_1(\Sigma_g)\longrightarrow B_n(\Sigma_g)$$

which fits into a commutative triangle of groups
\begin{displaymath}
    \xymatrix{ B_n(\Sigma_g) \ar[r]^{\omega} & H_1(\Sigma_g;\mathbb{Z}) \\
               \pi_1(\Sigma_g) \ar[u]^{\upsilon}\ar[ur]_{\Phi} }
\end{displaymath}
where $\Phi$ denotes the Hurewicz homomorphism. Since $\Phi$ is surjective, we immediately see that $\omega$ is surjective as well.

\subsection{Braids from Triangulations}
Let $K$ be a finite simplicial complex with geometric realization $|K|$ and let $\theta:|K|\xrightarrow{\simeq}\Sigma_g$ be an arbitrary triangulation of $\Sigma_g$. We fix this triangulation throughout the paper. We will identify $|K|$ with its image in $\Sigma_g$ under $\theta$, so that in particular, a ``vertex" in the triangulation of $\Sigma_g$ refers to a 0-simplex of $K$. Let the basepoint of $B_n(\Sigma_g)$ be the vertex set $K_0$ of our triangulation.
Fix some directed edge $e=(v_0, v_1)$ in the triangulation of $\Sigma_g$ which is a 1-face of two 2-simplices, say $(v_0, v_1, v_2)$ and $(v_0, v_1, v_3)$. We can obtain a braid by rotating $v_0$ and $v_1$ clockwise around $e$ until $v_0$ and $v_1$ have swapped positions whilst remaining in the interior of $(v_0, v_1, v_2)\cup(v_0, v_1, v_3)$ (see Figure \ref{edge}). All the strands starting at points in $K_0 -\{v_0, v_1\}$ remain constant.
\begin{figure}\label{edge}
 \centering
\begin{tikzpicture}

\draw[gray, ultra thick] (0,0) -- (3,0) -- (3,3) -- cycle;
\draw[gray, ultra thick] (0,0) -- (0,3) -- (3,3) -- cycle;
\draw[->] (0,0) .. controls (0.5,2.5) .. (2.8, 2.9);
\draw[->] (3,3) .. controls (2.5,0.5) .. (0.2, 0.1);
 
\end{tikzpicture}
\caption{Edge Braid on a Local Piece of a Triangulation}
\end{figure}
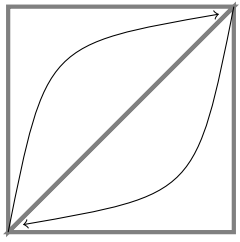
Braids constructed in this fashion are called \textit{edge braids}. Every edge $e=(v_0, v_1)$ in $\Sigma_g$ yields two mutually inverse edge braids, which we denote by $b_e$ (the ``clockwise" edge braid) and $b_e^{-1}$ (the ``counter-clockwise" edge braid). Let $E_n^{\theta}(\Sigma_g)$ denote the subgroup of $B_n(\Sigma_g)$ generated by the edge braids corresponding to the triangulation $\theta$. If the context is clear, we will write $E_n(\Sigma_g)$ instead of $E_n^{\theta}(\Sigma_g)$. We refer to elements of $E_n(\Sigma_g)$ as \textit{quasi-edge braids}. Note that each edge braid vanishes under $\omega$, so that there is an inclusion $E_n(\Sigma_g)\subset\mathrm{ker}(\omega)$.\\
\indent An \textit{edge path of length $k$} in $K$ is a concatenation of directed edges $\lambda=e_1 
\ast e_2\ast\ldots\ast e_k$ in the triangulation of $\Sigma_g$ such that the target of $e_i$ is the source of $e_{i+1}$ for $1\leq i\leq k-1$. $\lambda$ is called an \textit{edge loop} if the target of $e_k$ is the source of $e_1$. We will abuse terminology by referring to edge paths in $K$ as edge paths in $\Sigma_g$. An edge path/loop is \textit{simple} if it is non-self intersecting. A ``vertex" in $\lambda$ refers to a vertex of one of the edges contained in $\lambda$.\\
Recall that for any $v\in K$ there are isomorphisms $E(K,v)\simeq\pi_1(|K|,v)\simeq\pi_1(\Sigma_g,v)$, where $E(K,v)$ denotes the \textit{edge path group} of $K$. Thus, we can naturally consider (edge-equivalence classes of) edge loops in the triangulation of $\Sigma_g$ based at $v$ as (homotopy classes of) loops also based at $v$.

\section{Main Results}
Our main result is Theorem \ref{theorem}, which, given any triangulation of $\Sigma_g$, gives a characterization of $\mathrm{ker}(\omega)$ using edge braids.
\begin{thm}\label{theorem}
Let $\Sigma_g$ be equipped with an arbitrary simplicial triangulation $|K|\xrightarrow{\simeq}\Sigma_g$ and let $n=\# K_0$. Let the surface braid group $B_n(\Sigma_g)$ have basepoint $K_0$. Then the kernel of the total winding number map $\omega:B_n(\Sigma_g)\rightarrow H_1(\Sigma_g; \mathbb{Z})$ is precisely $E_n(\Sigma_g)$.
\end{thm}
Since $\omega$ is surjective, Theorem \ref{theorem} implies that any triangulation of $\Sigma_g$ induces a short exact sequence of groups
$$
1\longrightarrow E_n(\Sigma_g)\longrightarrow B_n(\Sigma_g)\longrightarrow H_1(\Sigma_g;\mathbb{Z})\longrightarrow 0.
$$
\begin{cor}
\textit{$E_n(\Sigma_g)$ contains the commutator subgroup $[B_n(\Sigma_g),B_n(\Sigma_g)]$.}
\end{cor}
\begin{proof}
Theorem \ref{theorem} supplies an isomorphism of groups $H_1(\Sigma_g;\mathbb{Z})\simeq B_n(\Sigma_g)/E_n(\Sigma_g)$. Since $H_1(\Sigma_g;\mathbb{Z})$ is abelian, we have that $[B_n(\Sigma_g),B_n(\Sigma_g)]\subset E_n(\Sigma_g)$.
\end{proof}
\begin{rem}
By Heawood's bounds [3], the number of vertices $n$ of the simplicial complex $K$ used in any triangulation of $\Sigma_g$ for $g\neq 2$ must satisfy  
\begin{equation}
    n\geq \frac{7+\sqrt{49-24\chi(\Sigma_g)}}{2}=\frac{7+\sqrt{1+48g}}{2}
\end{equation}
where $\chi(\Sigma_g)=2-2g$ denotes the Euler characteristic of $\Sigma_g$. In particular, Theorem \ref{theorem} may apply for $n$ sufficiently large relative to the genus $g$.
\end{rem}

The remainder of this paper is dedicated to proving Theorem \ref{theorem}. The genus $g=0$ case is well-known; we provide a proof for completeness.
\subsection{Proof of Theorem \ref{theorem} for $g=0$}\label{prooooof}
Since $H_1(\Sigma_0;\mathbb{Z})\simeq0$, the genus $0$ case is the statement that $B_n(\Sigma_0)=E_n(\Sigma_0)$ given any triangulation of $\Sigma_0$, where $n$ is the number of vertices in the triangulation. Recall that $B_n(\Sigma_0)$ is generated by the ``transposition" braids $\sigma_i$ for $1\leq i\leq n-1$ analogous to the generators of Artin's braid group $B_n$. Thus, it suffices to show that each $\sigma_i$ is a quasi-edge braid. Let $x_i$ and $x_{i+1}$ be the elements of the basepoint of $B_n(\Sigma_0)$ that $\sigma_i$ swaps. Fix a simple edge path $\delta=e_{i_1}\ast\ldots\ast e_{i_n}$ from $x_i$ to $x_{i+1}$. Then the quasi-edge braid
\begin{equation}\label{cool}
q_{\delta}=b_{e_{i_1}}b_{e_{i_2}}\ldots b_{e_{i_n}}b_{e_{i_{n-1}}}^{-1}b_{e_{i_{n-2}}}^{-1}\ldots b_{e_{i_1}}^{-1}
\end{equation}
is equal to $\sigma_i$ (it swaps $x_{i}$ and $x_{i+1}$ while leaving all other elements of the basepoint fixed), so we're done.\\
For any edge path $p$ in the triangulation of $\Sigma_g$, $q_{p}$ denotes the quasi-edge braid constructed in manner of Equation (5).
\begin{exmp}
The most basic example of of the genus $0$ case of Theorem \ref{theorem} is when the triangulation is the canonical homeomorphism $\partial(\Delta^3)\simeq\Sigma_0$, where $\partial(\Delta^3)$ denotes the boundary of the standard 3-simplex. By the assumptions of the theorem, the elements of the basepoint of $B_4(\Sigma_0)$ are the four endpoints of $\Delta^3$, all of which are pairwise adjacent. Hence, each generator of $B_4(\Sigma_0)$ is actually an edge braid.
\end{exmp}
\subsection{Reduction of Theorem \ref{theorem} to Pure Balanced Braids on $\Sigma_g$}
We start with the following observation.
\begin{prop}\label{surj}
The restriction $f|_{E_n(\Sigma_g)}:E_n(\Sigma_g)\rightarrow S_n$ is surjective.
\end{prop}
\begin{proof}
It suffices to show each transposition $s$ in $S_n$ is hit by $f|_{E_n(\Sigma_g)}$. Let $(x_1, \ldots, x_n)$ be the basepoint of $B_n(\Sigma_g)$. The proof is the same as in the $g=0$ case. Explicitly, let $s$ swap $i$ and $j$, where we assume without loss of generality that $1\leq i< j\leq n$. Fix a simple edge path $\lambda$ from $x_i$ to $x_j$. Then $f(q_{\lambda})=s$, so we're done.
\end{proof}

Let $l:S_n\longrightarrow\mathbb{Z}_{\geq 0}$ denote the \textit{length function} of $S_n$ relative to the generating transpositions $s_i\in S_n$ of the usual Coxeter presentation. $l(\gamma)$ is defined to be the minimum number of transpositions required to express the permutation $\gamma\in S_n$. The following fact now is an easy consequence of Proposition \ref{surj}. 
\begin{prop}\label{gen}
Every element in $B_n(\Sigma_g)$ can be written as a product of pure braids and edge braids.
\end{prop}
\begin{proof}
Define a function $$l_{\star}:B_n(\Sigma_g)\longrightarrow\mathbb{Z}_{\geq 0}$$ $$[\sigma]\mapsto l(f([\sigma])).$$ This is clearly well-defined. We induct on $l_{\star}([\sigma])$. The base case $l_{\star}([\sigma])=0$ is clear since $[\sigma]$ must be a pure braid. Suppose the result holds for all $[\sigma]$ with $l_{\star}([\sigma])=k$. To complete the inductive step, it suffices to show that any braid $[\sigma]$ of length $k+1$ can be multiplied by some quasi-edge braid such that the resulting braid has length $k$. Write $f([\sigma])=s_{i_1} s_{i_2}\ldots s_{i_{k+1}}$ for transpositions $s_{i_j}\in S_n$, $1\leq j\leq k+1$. By Proposition \ref{surj}, we may choose some quasi-edge braid $[\mu]$ such that $f([\mu])=s_{i_{k+1}}^{-1}$. Thus, $l_{\star}([\sigma\mu])=k$, which completes the proof.
\end{proof}
By Proposition \ref{gen}, it is sufficient to show that the subgroup $\mathrm{ker}(\omega|_{P_n(\Sigma_g)})=P_n(\Sigma_g)\cap\mathrm{ker}(\omega)$ of $\mathrm{ker}(\omega)$ is generated by edge braids in order to deduce Theorem \ref{theorem}. This is a useful reduction since we may think of pure braids as collections of homotopy classes of loops on $\Sigma_g$.

\subsection{Some Properties of $E_n(\Sigma_g)$}\label{braids}
We prove some results concerning which braids on $\Sigma_g$ are quasi-edge braids and describe some relations that the edge braids satisfy.

\subsubsection{Conjugation action of $E_n(\Sigma_g)$}\label{conjj}
We briefly describe a property of the conjugation action of certain quasi-edge braids which is relevant to the proof of Lemma \ref{super}. In particular, conjugation by certain quasi-edge braids has a useful property when the pure braid being conjugated has exactly one non-trivial vertex loop.  Let $\varphi:B_n(\Sigma_g)\rightarrow\mathrm{Aut}P_n(\Sigma_g)$ denote the conjugation homomorphism and let $[\gamma]\in P_n(\Sigma_g)$ be a pure braid with exactly one non-trivial vertex loop, say $[\gamma]_{x_i}$ for some $i\in\{1, \ldots, n\}$. \\
 \indent Let $x_j$ be a vertex adjacent to $x_i$ in the triangulation of $\Sigma_g$ and let $e$ be an edge connecting them. By a direct computation, we see that the braid $b_{e}[\gamma]b_{e}^{-1}$ still has exactly one non-trivial vertex loop, except that it is located at $x_j$ instead of $x_i$, so that conjugating by $b_{e}$ ``moves" the loop at $x_i$ to $x_j$. Suppose that $x_i$ and $x_j$ are vertices in the triangulation that are not necessarily adjacent. Fix a simple edge path $\eta$ from $x_i$ to $x_j$. Extrapolating from the above case, we see that $[\gamma]$ conjugated by the quasi-edge braid $q_\eta$ (constructed in the fashion of Equation (5) of Section \ref{prooooof}) has exactly one non-trivial vertex loop located at $x_j$.\\
\indent For any braid $[\alpha]\in B_n(\Sigma_g)$, let $C_{[\alpha]}$ denote the subset of the basepoint $\{x_1, \ldots, x_n\}$ consisting of the points whose induced vertex loops are trivial. Since $[\gamma]$ has only one non-trivial vertex loop, it can naturally be regarded as element of $\pi_1(\Sigma_g-C_{[\gamma]},x_i)$. We will not make a distinction between such braids and elements of $\pi_1(\Sigma_g-C_{[\gamma]},x_i)$ for the rest of the paper. Similarly, $q_\eta[\gamma]q_\eta^{-1}$ can be seen as an element of $\pi_1(\Sigma_g-C_{{q_{\eta}[\gamma]q_{\eta}^{-1}}},x_j)$. The above discussion can be re-phrased via the following proposition.

\begin{prop}
Let $x_i$ and $x_j$ vertices in the triangulation of $\Sigma_g$ and let $\lambda$ be a simple edge path from $x_i$ to $x_j$. Then the conjugation map $\varphi(q_{\lambda})\in\mathrm{Aut}B_n(\Sigma_g)$ restricts to an isomorphism $$\pi_1(\Sigma_g-C_{[\gamma]},x_i)\xrightarrow{\simeq}\pi_1(\Sigma_g-C_{q_{\lambda}[\gamma]q_{\lambda}^{-1}},x_j)$$
\end{prop}
\subsubsection{Quasi-Edge Braid Constructions and Relations}

\begin{lem}\label{super}
Let $\Lambda=e_1\ast \ldots\ast e_k$ be a simple edge loop in $\Sigma_g$ and fix two vertices $v_i, v_j$ contained in $\Lambda$. Then there exists a quasi-edge braid $w(v_i,v_j)$ such that the vertex loop $w_{v_i}$ is homotopic to $\Lambda$ and $w_{v_j}^{-1}$ is homotopic to $\Lambda$.
\end{lem}
\begin{proof}
For ease of notation, the indices of the vertices and edges in $\Theta$ will be taken modulo $k$ (i.e $v_k=v_0$).
Let $e_i=(v_i, v_{i+1})$ and let $\Omega$ denote the quasi-edge braid
\begin{equation}
    \Omega=b_{e_i}b_{e_{i+1}}\ldots b_{e_{i-2}}b_{e_{i-1}}.
\end{equation} 
It is clear that $\Omega_{v_i}$ is homotopic to $\Lambda$. Then $\widehat{\Omega}=\Omega b_{e_{i-1}}b_{e_{i-2}}\ldots b_{e_{i+2}}b_{e_{i+1}}$ is such that $\widehat{\Omega}_{v_{i+1}}^{-1}$ is also homotopic to $\Lambda$. Furthermore, the only non-trivial vertex loops of $\widehat{\Omega}$ are located at $v_i$ and $v_{i+1}$. Let $\mu$ be the unique edge path from $v_{i+1}$ to $v_j$ that does not contain $v_i$ and is a subset of $\Lambda$. Then $w(v_i, v_j)=q_{\mu}\widehat{\Omega}q_{\mu}^{-1}$ gives the desired braid via the discussion in Section \ref{conjj}.


\end{proof}



\begin{lem}\label{cool}
Let $\lambda =e_1\ast \ldots\ast e_k$ be a simple edge path from $v_1$ to $v_{k+1}$, where $e_i=(v_i, v_{i+1})$ for $1\leq i\leq k$. Then there exists a quasi-edge braid such that the induced vertex loop at $v_1$ has a winding number of one about $v_{k+1}$ and zero around any other element in the basepoint of $B_n(\Sigma_g)$ (which are precisely the other vertices in the triangulation).

\end{lem}

\begin{proof}
The quasi-edge braid 
\begin{equation}
w_{\lambda}=b_{e_1}b_{e_2}\ldots b_{e_{k-1}}b_{e_k}^{-2}b_{e_{k-1}}\ldots b_{e_2}b_{e_1}
\end{equation}
gives the desired braid. The fact that $w_{\lambda}$ has trivial winding number around any other vertex follows from the definition of edge braids.
\end{proof}
\begin{rem}
Note that $w_{\lambda}$ can be viewed as a non-identity element of $\pi_1(\Sigma_g-v_{k+1},v_1)$. It is constructed so that set $\{w_{\lambda}\}\cup S$ generates $\pi_1(\Sigma_g-v_{k+1},v_1)$, where $S$ is the set of generators of $\pi_1(\Sigma_g,v_1)$ (viewed on the punctured surface $\Sigma_g-v_{k+1}$). The quasi-edge braids built in Lemmata \ref{super} and \ref{cool} are used in Proposition \ref{one}.
\end{rem}
\begin{lem}
Let $\alpha=(v_0, v_1, v_2)$ and $\alpha'=(v_0, v_2, v_3)$ be two 2-simplices in the triangulation such that $\alpha\cap\alpha'=(v_0, v_2)$. Let $e_{i}=(v_i,v_{i+1})$ for $0\leq i\leq 2$ and $e_{3}=(v_3,v_{0})$. Then there exists a quasi-edge braid that is homotopic to $e_0\ast e_1\ast e_2\ast e_3$.
\end{lem}
\begin{proof}
The braid $b_{e_0}b_{e_1}b_{e_0}b_{e_2}b_{e_0}b_{e_2}$ gives the desired braid.
\end{proof}

\begin{lem}{(Local Edge Braid Relations).}\label{local}
Fix a 2-simplex $\alpha=(v_0, v_1, v_2)$ in the triangulation of $\Sigma_g$ with boundary $\partial(\alpha)=\{e_0, e_1, e_2\}$, where $e_0=(v_0,v_{1})$, $e_1=(v_1,v_{2})$, and $e_2=(v_2,v_{0})$. Then the following relations hold:
\begin{equation}
b_{e_1}b_{e_0}=b_{e_2}b_{e_1}=b_{e_0}b_{e_2}
\end{equation}
\begin{equation}
b_{e_0}b_{e_1}b_{e_0}=b_{e_1}b_{e_0}b_{e_1}
\end{equation}
\begin{equation}
b_{e_1}b_{e_2}b_{e_1}=b_{e_2}b_{e_1}b_{e_2}
\end{equation}
\begin{equation}
b_{e_0}b_{e_2}b_{e_0}=b_{e_2}b_{e_0}b_{e_2}
\end{equation}

\end{lem}

\begin{proof}
Relation (8) follows from a direct computation. By multiplying both sides of $b_{e_1}b_{e_0}=b_{e_2}b_{e_1}$ by $b_{e_0}$ on the left and using the relation $b_{e_0}b_{e_2}=b_{e_1}b_{e_0}$, we deduce relation (9):
$$
b_{e_0}b_{e_1}b_{e_0}=b_{e_0}b_{e_2}b_{e_1}
$$
$$
=b_{e_1}b_{e_0}b_{e_1}.
$$
Relations (10) and (11) follow by symmetry.
\end{proof}

\begin{rem}
The relations between edge braids described in Lemma \ref{local} are very similar to those between the generators of Artin's classical braid groups $B_n$, which themselves induce relations between the generating transpositions in the Coxeter presentation of $S_n$. 

\end{rem}

\subsection{Proof of Theorem \ref{theorem} for $g\geq 1$}
\subsubsection{Outline of Approach}\label{out}
Fix a arbitrary element $[\sigma]\in P_n(\Sigma_g)\cap\mathrm{ker}(\omega)$. We develop a procedure to successively multiply $[\sigma]$ by quasi-edge braids until the resulting braid $[\sigma']$ has at most one non-trivial vertex loop at some vertex $x_i$ in the triangulation. This makes use of the results of sections \ref{braids} and \ref{absd}. The braid $[\sigma']$ can then be regarded as an element of $\pi_1(\Sigma_g-C_{[\sigma']},x_i)$. We then use topological and group theoretical methods involving the fundamental groups of (punctured) surfaces to deduce that $[\sigma']$ itself is a quasi-edge braid, which implies the result.

\subsubsection{Generators of $\pi_1(\Sigma_g)$ and Edge Loops}\label{absd}

Proposition \ref{hom} and Lemma \ref{unzip} allow us to conveniently describe generators of $\pi_1(\Sigma_g)$ using edge loops. Recall that $\pi_1(\Sigma_g)$ has generators $ [f_1], [f_2], \ldots, [f_{2g}]$.
If we think of $\Sigma_g$ as the connected sum $\Sigma_g=\Sigma_1\#\ldots\#\Sigma_1$ of $g$ tori, then $[f_{i}]$ and $[f_{i+1}]$ for $i \equiv 1 \Mod{2}$ can be realized as generators of the fundamental group of the $i^{\mathrm{th}}$ torus in the connected sum. We refer to the $[f_{i}]$'s as the \textit{standard generators} of $\pi_1(\Sigma_g)$.

\begin{prop}\label{hom}
Let $v\in\Sigma_g$ be a vertex in the triangulation of $\Sigma_g$. Fix the usual generators $[f_1], \ldots, [f_{2g}]$ of $\pi_1(\Sigma_g,v)$. Then there exists representatives of each class $[f_i]$ for $1\leq i\leq 2g$ that are simple edge loops in the triangulation of $\Sigma_g$, i.e each $f_i$ is homotopic to a simple edge loop.
\end{prop}

\begin{proof}
Clearly we can assume that $f_i$ is a simple loop. Locally deform $f_i$ to a homotopic loop $f_i'$ based at $v$ such that the only vertex in the triangulation of $\Sigma_g$ that $\mathrm{im}(f_i')$ intersects is $v$ and $f_i'$ remains simple. Let $S$ denote the set of simplices $\beta$ in the triangulation such that $\beta\cap\mathrm{im}(f_i')$ is nonempty. Let $S (v)$ denote the star of $v$, i.e the set of simplices in the triangulation that contain $v$ as $0$-face. It is clear by inspection that $\mathrm{im}(f_i')$ intersects exactly two elements of $S(v)$ (or that $f_i'$ can be homotoped into such a loop whose image satisfies this), which we denote by $\gamma_1$ and $\gamma_2$. Further homotope $f_i'$ into a simple loop $f_i''$ such that $f_i''$ intersects precisely two 1-faces of each element of $\beta-\{\gamma_1,\gamma_2\}$. Let $S'$ denote the set of all 1-faces of elements in $\beta-\{\gamma_1,\gamma_2\}$ that are disjoint from $\mathrm{im}(f_i')$. We construct an algorithm which helps us build an edge loop which is homotopic to $f_i''$.
\begin{itemize}
    \item Step $0$: Choose a vertex $w_1$ in $S(v)$ that is also an endpoint of an element of $S'$
    \item Step $1$: Let $e_1$ denote the unique element of $S'$ of which $w_1$ is an endpoint of. Set $P_1=\{e_1\}$.
    \item Step $2$: Let $e_2$ denote the unique element of $S'-\{e_1\}$ that contains as an endpoint the endpoint of $e_1$ that is not $w_1$. Let $w_2$ denote this endpoint and set $P_2=\{e_1,e_2\}$
    \item Step $3$: Continue in the same fashion as Steps 1 and 2 by letting $e_j$ denote the unique element of $S'-\{e_1, \ldots, e_{j-1}\}$ that contains as an endpoint the endpoint of $e_{j-1}$ that is not $w_{j-1}$, letting $w_{j}$ denote this endpoint and setting $P_k=P_{k-1}\cup \{e_j\}$. Repeat until $w_{j+1}$ (the endpoint of $e_j$ that is not $w_j$) is a vertex in $S(v)$ for $j\geq 2$.
    \end{itemize}
Since $w_i$ is always a $0$-face of a 2-simplex in the triangulation which intersects $\mathrm{im}(f_{i}'')$, the process in Step 3 will terminate, say after $k$ total iterations. Since $w_1, w_{k+1}\in S(v)$, there are edges $r=(w_{k+1},v)$ and $r'=(v,w_1)$. Then $l=e_1\ast e_2\ast\ldots\ast e_k\ast r\ast r'$ (the concatenation of the elements in $P_{k}$ with $r\ast r'$) is a simple edge loop which is homotopic to $f_{i}''$, and hence to $f_{i}$. That $l$ is simple follows from the fact that $f_i''$ is simple, so the proposition follows.
\end{proof}
The next lemma lets us ``extend" edge loops to homotopic edge loops that intersect certain vertices. This is a key ingredient in the proof of Proposition \ref{one}.
\begin{lem}\label{unzip}
Fix vertices $v$, $v'$ in the triangulation and let $L$ be a simple edge loop. Then $L$ is homotopic to a simple edge loop that intersects $v$ and $v'$.
\end{lem}
\begin{proof}
We show that $L$ can be homotoped so that it intersects $v$ (a similar argument proves the full lemma). Choose a simple edge path $\lambda$ with endpoints $v$ and some vertex $x$ contained in $L$. Let $k$ be the length of $\lambda$. 
By induction, it suffices to show that there exists an edge loop $\xi$ homotopic to $L$ such that there is a simple edge path from $v$ to some vertex in $\xi$ of length less than $k$. Let $e_1, \ldots, e_r$ denote ordered list of edges that defines $L$. Let $\lambda'$ denote the edge path $e_1, \ldots e_{\floor{\frac{r}{2}}}$ and let $\lambda''$ denote the edge path $e_{\ceil{\frac{r}{2}}}, \ldots, e_r$. Let $E(x)$ denote the set of edges in the triangulation that contain $x$ as an endpoint. Denote the unique edge in $E(x)\cap\lambda'$ by $\mu_0$. Label the elements of $E(x)$ that lie in between $\lambda'$ and $\lambda$ in sequential order starting with $\mu_1=e_1\in\lambda'$ and ending at $\mu_p\in\lambda$, so that there are no edges in $E(x)$ that are between $\mu_i$ and $\mu_{i+1}$. 
\begin{claim}
The edges $\mu_i$ and $\mu_{i+1}$ must be 1-faces of a common 2-simplex $\sigma_{i,i+1}$.
\end{claim}
\begin{proof}
Suppose not. Since there are no edges in $E(x)$ between $\mu_i$ and $\mu_{i+1}$, this would imply that $\Sigma_g$ is homotopy equivalent to a wedge sum of $2g$ circles, which is a contradiction.
\end{proof}

For such $\mu_i$ and $\mu_{i+1}$, let $\mu_{i,i+1}$ denote the 1-face of $\sigma_{i,i+1}$ that is not $\mu_i$ or $\mu_{i+1}$. Let $w_i$ denote the endpoint of $\mu_i$ that is not $x$. We proceed by casework on the configuration of edges in $E(x)$.
\begin{itemize}
    \item Case 1: Suppose that there does not exist an edge $\mu_{j}$ (with $2\leq j\leq p-1$) such that $w_j\in\lambda\cup\lambda'\cup\lambda''$ (See Figure 3). Let $g_1$ denote the unique edge in $E(x)\cap\lambda$. Set $\xi=g_1\ast\mu_{p-1, p}\ast\mu_{p-2, p-1}\ast\ldots\ast\mu_{1,2}\ast e_1\ast\ldots \ast e_{k}$; this is homotopic to $L$ since the subcomplex $\sigma_{1,2}\cup\ldots\cup\sigma_{p-1,p}$ is contractible. There is clearly a sub-path of $\lambda$ from $v$ to $w_p\in\xi$ of length $k-1$, so this completes the induction.
    \end{itemize}
    In the following cases, we assume that there exists an edge $\mu_j$ (with $2\leq j\leq p-1$) such that $w_j\in\lambda\cup\lambda'\cup\lambda''$. Let $i$ be the smallest index such that $\mu_i$ satisfies this condition.
    \begin{itemize}
    \item Case 2: $w_i\in\lambda$ (See Figure 4). The argument in the proof of Case 1 holds.
    \vspace{2mm}
  
    \vspace{2mm}
    \item Case 3: $w_i\in\lambda''$ (See Figure 5). Note that for all $j$ such that $i+1\leq j\leq p-1$, we must have that $w_{j}\in\lambda\cup\lambda''$. There are two possible subcases.
    \begin{itemize}
    \item Subcase 3.1: There exists a $j$ such that $i+1\leq j\leq p-1$ and $w_j\in \lambda$. Then there is some edge path $\Delta$ from $w_i$ to $w_j$. Let $\Delta'$ denote the unique sub-path of $\lambda$ with endpoints $w_j$ and $x$. Let $e_k$ denote the unique edge in $L$ such the head of $e_k$ is $w_i$. Then the edge loop $\Delta'\ast e_1\ast e_2\ast\ldots\ast e_k\ast\Delta$ is homotopic to $L$ and there is evidently a proper sub-path of $\lambda$ from $x$ to $v$.
        \vspace{2mm}
        \item Subcase 3.2: No such $j$ exists. 
        \begin{itemize}
            \item 3.2a: There is some $n$ such that $i+1\le n\leq p-1$ and $w_n\in \lambda''$. Let $k$ be the maximum index such that $i+1\leq k\leq p-1$ and $w_k\in\lambda''$. Then there is an edge path $\Upsilon$ from $w_k$ to $w_p$. Let $e_t$ be the unique edge in $f_i$ such that the head of $e_t$ is $w_k$. Then the edge loop $\Upsilon\ast\mu_p\ast e_1\ast e_2\ast \ldots\ast e_t$ satisfies the desired conditions.

            \item 3.2b: No such $n$ exists. Let $\zeta$ denote the unique edge in $E(x)\cap\lambda''$. Since none of the edges in $E(x)$ that lie between $\mu_p$ and $\zeta$ can have an endpoint in $\lambda'$, the proof reduces to the same arguments given in Case 2 and Case 4 (see below), except they are applied to the edges in $E(x)$ that lie between $\mu_p$ and $\zeta$ instead of the edges in $E(x)$ that lie between $\mu_1$ and $\mu_p$. 
        \end{itemize}

        \end{itemize}
    \vspace{2mm}
     \item Case 4: $w_i\in\lambda'$ (See Figure 6).
    \begin{itemize}
    \item Subcase 4.1: For all $i$ with $i+1\leq j\leq p-1$, $w_i\notin \lambda\cup\lambda'\cup\lambda''$. This is handled the same way as Case 1.
    \item Subcase 4.2: There is some $i$ such that $i+1\leq j\leq p-1$ and $w_i\in \lambda\cup\lambda'\cup\lambda''$. In the remaining subcases, $q$ denotes the largest index such that $i+1\leq q\leq p-1$ and $w_q\in\lambda\cup\lambda'\cup\lambda''$.
    \begin{itemize}
        \item Subcase 4.2a: $w_q\in\lambda$. Clearly, $w_q\neq w_p$ since the equality $w_q= w_p$ would contradict the definition of a simplicial complex. Let $\chi$ be the edge path obtained by concatenating  $\mu_q$ with the unique sub-path of $\lambda$ with endpoints $w_q$ and $v$. Then $\chi$ is a simple edge path from $x$ to $v$ that is shorter than $\lambda$, as desired.
        \item Subcase 4.2b: $w_q\in\lambda'$. Then for any $m$ with $i+1\leq m\leq q-1$, it must be that  $w_{m}\in\lambda'$ or $w_{m}\notin\lambda\cup\lambda'\cup\lambda''$. By the same logic as the proof of Subcase 3.1, there is an edge path $\gamma$ from $w_p$ to $w_q$. Let $e_s$ denote the unique edge in $L$ with tail $w_q$. Then by similar logic to the proof of Case 1, we see that the edge loop $\mu_p\ast\gamma\ast e_s\ast e_{s+1}\ast\ldots\ast e_r$ works.
        \item Subcase 4.2c: $w_q\in\lambda''$. This is handled using arguments similar to those in Subcase 3.1
        
         \end{itemize}
        \end{itemize}
        

\end{itemize}

All cases are covered, so the proof is complete.
\end{proof}
\begin{center}
\begin{figure}
\includegraphics[width=0.7\textwidth]{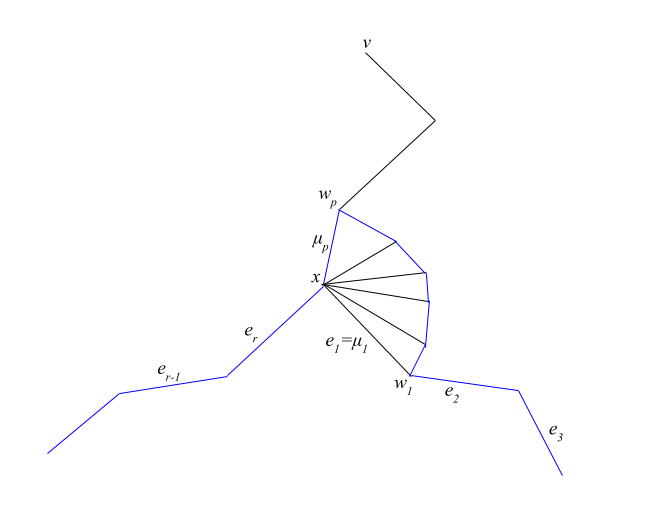}
\caption{Illustration of Case 1 of Lemma \ref{unzip}. The blue edge path is a portion of the homotoped edge loop $\xi$.}
\end{figure}

\begin{figure}
\includegraphics[width=0.7\textwidth]{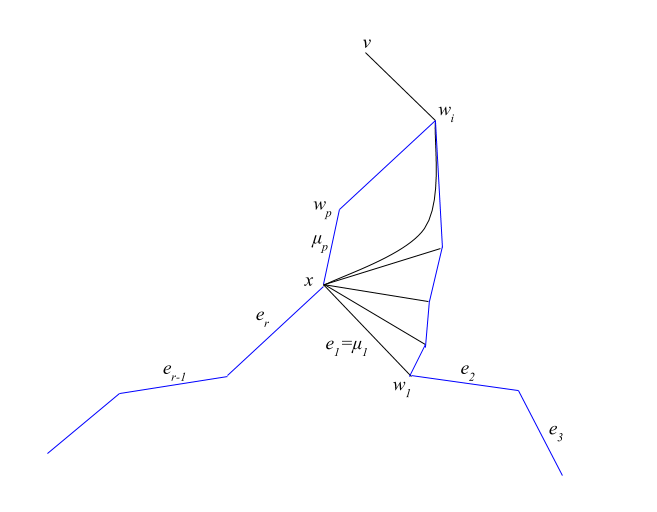}
\caption{Illustration of Case 2 of Lemma \ref{unzip}}
\end{figure}

\begin{figure}
\includegraphics[width=0.7\textwidth]{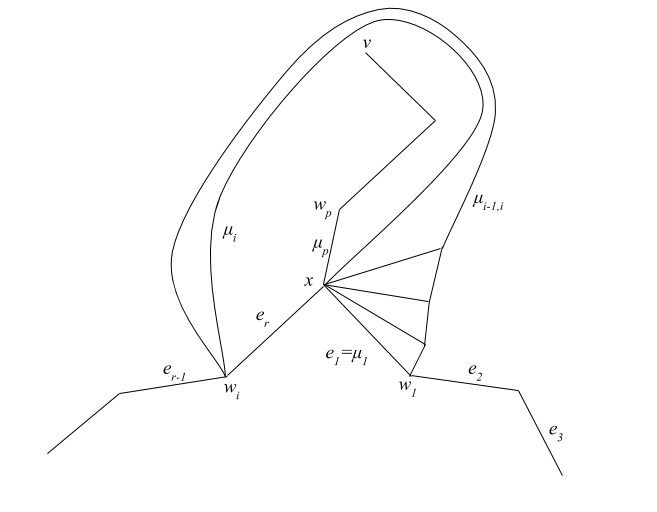}
\caption{Illustration of Case 3 of Lemma \ref{unzip}}
\end{figure}
\begin{figure}
\includegraphics[width=0.7\textwidth]{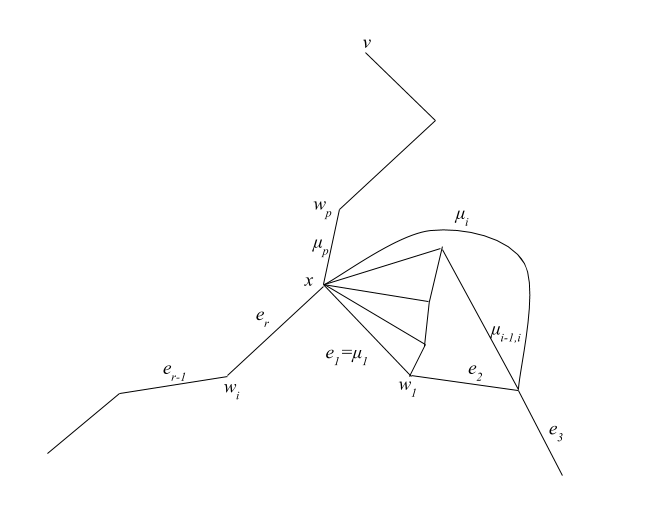}
\caption{Illustration of Case 4 of Lemma \ref{unzip}}
\end{figure}
\end{center}

\subsubsection{Words}
We prove a combinatorial lemma which is used in the next section. Throughout this subsection, $G$ denotes a finitely generated group with generators $a_1, \ldots, a_n$. By a \textit{word} in $F_n$, we mean a potentially unreduced word (a string consisting of generators in which terms such as $xx^{-1}$ need not be simplified to $1$). Let $S\subset\{a_1, \ldots a_n\}$ be a subset. Fix a word $w=a_{i_1}\ldots a_{i_k}$ in $G$. Let $w_{S}^\mathrm{min}$ (resp. $w_{S}^\mathrm{max}$) denote the minimum (resp. maximum) index $m$ such that $a_{i_m}\in S$. A word $w$ is \textit{$S$-connected} if $a_{i_j}\in S$ for all $w_{S}^\mathrm{min}\leq j\leq w_{S}^\mathrm{max}$. 
\begin{prop}\label{represent}
Let $w$ be a word in $G$. Then for any non-empty ordered subset $S\subset\{a_1, \ldots a_n\}$, there exists an $S$-connected word $w_{\star}$ such that the following conditions hold.
\begin{itemize}
    \item All elements in $w_{\star}$ that are also in $S$ appear in the same order that they appear in $w$.
    \item $w=w_{\star}$.
    \item All elements in $w_{\star}$ are either elements of $w$ or commutators of elements in $w$.
\end{itemize}

\end{prop}

\begin{proof}
We induct on the length $l(w)$ of $w$. The base case $k=1$ is trivial. Suppose the result holds for all words of length less than $k$ (for $k>1$) and let $w=a_{i_1}\ldots a_{i_k}$ be a word of length $k$. Let $m=w_{S}^\mathrm{min}$ and set $r=a_{i_{m+1}}\ldots a_{i_{k}}$. Then by the inductive hypothesis there exists an $(S-\{a_{i_{m}}\})$-connected word $r_{\star}$ such that $r=r_{\star}$ and all elements in $r_{\star}$ that are also in $S-\{a_{i_{m}}\}$ appear in the same order that they appear in $r$. Consider the word $w'=a_{i_1}\ldots a_{i_{m}} r_{\star}$. Let $b_{i_{m+j}}$ denote the element in the $j^{th}$ position of $r_{\star}$, so that $w'=a_{i_1}\ldots a_{i_{m}}b_{i_{m+1}}\ldots b_{i_{k}}$. Let $u$ denote the minimum index such that $b_{i_{u}}\in S-\{a_{i_{m}}\}$. Then we have that the word $$w_{\star}=a_{i_1}\ldots a_{i_{m-1}}[a_{i_{m}},b_{i_{m+1}}\ldots b_{i_{u-1}}]b_{i_{m+1}}\ldots b_{i_{u-1}}a_{i_{m}}b_{i_{u}}b_{i_{u+1}}\ldots b_{i_{k}}$$ satisfies the desired conditions, so the proof is complete.
\end{proof}
\begin{exmp}
Let $w=a_1 a_3 a_2 a_5 a_4$ and $S=\{a_2, a_4\}$. Then $w_{\star}=a_1a_3[a_2,a_5]a_5a_2a_4$ is $S$-connected and equal to $w$.
\end{exmp}

\subsubsection{Unwinding Pure Balanced Braids}
Recall that for any $[\sigma]\in B_n(\Sigma_g)$, $\pi_1(\Sigma_g-C_{[\sigma]})$ is the free group on the generators
$[f_1], [f_2], \ldots, [f_{2g+\#C_{[\sigma]}-1}]$. The generators $[f_i]$ for $1\leq i\leq 2g$ can be regarded as standard generators of $\pi_1(\Sigma_g)$, while the generators $[f_i]$ for $2g+1\leq i\leq 2g+\#C_{[\sigma]}-1$ have each wind around an element of $C_{[\sigma]}$ once. 

\begin{prop}\label{one}
Let $[\sigma]$ be a pure balanced braid such that $\# C_{[\sigma]}\leq n-2$. Fix $x_j\in \{x_1, \ldots, x_n\}-C_{[\sigma]}$ and let $[f_1], \ldots, [f_{2g+\#C_{[\sigma]}-1}]$ be the usual generators of $\pi_1(\Sigma_g-C_{[\sigma]}, x_j)$. Then for each $[f_i]$, $1\leq i\leq 2g+\#C_{[\sigma]}-1$, there exists a quasi-edge braid $[\gamma]$ such that the induced vertex loop of $\gamma$ on $x_j$ is $[f_i]$, while the induced vertex loops on all other vertices whose vertex loops under $[\sigma]$ were initially constant remain constant.
\end{prop}

\begin{proof}
There are two possible cases.
\begin{itemize}
\item 
Case 1: $1\leq i\leq 2g$. Then we can regard $[f_i]$ as a standard generator of $\pi_1(\Sigma_g)$; it is clear that there exists a representative $f_i$ of its homotopy class that has trivial winding number around all $y\in C_{[\sigma]}$. Since $k\leq n-2$, we can choose some $x_{k}\in \{x_1, \ldots, x_n\}-C_{[\sigma]}$ with $j\neq k$. Use Proposition \ref{hom} to find a simple edge loop $l$ which is homotopic to $e_i$ and passes through $x_j$ and $x_k$. Then take our desired quasi-edge braid to be the braid
$[\gamma]=w(x_j,x_k)$ constructed in  Lemma \ref{super}.
\vspace{2mm}
\item
Case 2: $2g+1\leq i\leq 2g+\#C_{[\sigma]}-1$. In this case, $f_i$ can be taken to be a vertex loop which has a winding number of one about some $x_m\in C_{[\sigma]}$ and zero with respect to all points in $C_{[\sigma]}-\{x_m\}$. Choose a simple edge path $\lambda$ from $x_j$ to $x_m$ in the triangulation of $\Sigma_g$. Then take $[\gamma]=w_{\lambda}$ to be the quasi-edge braid constructed in Lemma \ref{cool}.
\end{itemize}
\end{proof}
Fix $[\sigma]\in P_n(\Sigma_g)$ and write each vertex loop $[\sigma]_{x_i}$ for $x_i\in\{x_1, \ldots, x_n\}-C_{[\sigma]}$ as a product of the standard generators of $\pi_1(\Sigma_g - C_{[\sigma]})$. Using Proposition \ref{one}, we may sequentially multiply by $[\sigma]$ by quasi-edge braids whose vertex loops equal the inverses of the generators of $\pi_1(\Sigma_g - C_{[\sigma]})$ contained in the expression of $[\sigma]_{x_i}$ as a product of generators until it becomes the trivial vertex loop. This process can be repeated until all but one element of the basepoint $\{x_1, \ldots, x_n\}$ has a trivial vertex loop. From the discussion in Section \ref{out}, Theorem \ref{theorem} follows from Proposition \ref{two}.\\ 

\begin{lem}\label{split}
Let $[f_i]$ be such that $2g+1\leq i\leq 2g+\#C_{[\sigma]}-1\subset E_n(\Sigma_g)$. Then $[f_i]\in E_n(\Sigma_g)$
\end{lem}
\begin{proof}
Follows from the same argument given in Case 2 of Proposition \ref{one}.
\end{proof}
\begin{lem}\label{comm}
Commutators of the generators $[f_i]$ of $\pi_1(\Sigma_g-C_{[\sigma]},x)$ are quasi edge braids.
\end{lem}
\begin{proof}
Fix arbitrary generators $[f_i]$ and $[f_j]$ with $i\neq j$. We have the following possibilities:

\begin{itemize}
    \item Case 1: $1\leq i,j\leq 2g$. Once again, we may regard $[f_i]$ and $[f_j]$ as two of the standard generators of $\pi_1(\Sigma_g,x)$. By Proposition \ref{hom}, we may assume that $f_i$ and $f_j$ are simple edge loops that intersect only at the basepoint $x\in\Sigma_g$. Let $\lambda=e_1, \ldots, e_n$ (resp. $\lambda'=z_1, \ldots, z_k$) be the ordered list of edges constituting $f_i$ (resp. $f_j)$. For ease of notation, the indices of the vertices and edges in $\lambda$ (resp. $\lambda'$) will be taken modulo $n$ (resp. modulo $k$). Let $e_r=(v_r, v_{r+1})$ and $z_l=(w_l, w_{l+1})$. We may assume that $x=v_1=w_1$. Then a direct computation shows that the commutator
$
[b_{e_1}b_{e_2}\ldots b_{e_0},b_{z_1}b_{z_2}\ldots b_{z_0}]
$
is precisely $\big[[f_i],[f_j]\big]$. 
\vspace{2mm}
   \item Case 2: $2g+1\leq i,j\leq 2g+\#C_{[\sigma]}-1$. This is immediate from Lemma \ref{split}.
   \vspace{2mm}
    \item Case 3: $1\leq i\leq 2g$ and $2g+1\leq j\leq 2g+\#C_{[\sigma]}-1$. Let $f_i$ and $\lambda$ be as in Case 1. As before, $f_j$ can be taken to be a vertex loop which has a winding number of one about some $x_m\in C_{[\sigma]}$ and zero with respect to all points in $C_{[\sigma]}-\{x_m\}$. Fix a simple edge path $\mu$ from $x$ to $x_m$ and let $w_{\mu}$ be the quasi-edge braid constructed in Lemma $\ref{cool}$. Then the commutator
    $
    [b_{e_1}b_{e_2}\ldots b_{e_0},w_{\mu}]
    $
    is $\big[[f_i],[f_j]\big]$, as desired.
\end{itemize}
The remaining case follows by symmetry, so the proof is complete.

\end{proof}

\begin{rem}\label{qebbb}
Furthermore, one can show via a similar argument that \textit{conjugates} of commutators of the standard generators of $\pi_1(\Sigma_g-C_{[\sigma]},x)$ are quasi-edge braids. This implies the commutator subgroup $[\pi_1(\Sigma_g,x),\pi_1(\Sigma_g,x)]$ is contained in $E_n(\Sigma_g)$.
\end{rem}
\begin{prop}\label{two}
Let $[\sigma]$ be a pure balanced braid with $\#C_{[\sigma]}=n-1$, so that $[\sigma]$ has exactly one non-trivial vertex loop, say $[\sigma]_{x_j}$. Then $[\sigma]\in E_n(\Sigma_g)$.
\end{prop}

\begin{proof}
Regard $[\sigma]$ as an element of $\pi_1(\Sigma_g - C_{[\sigma]}, x_j)$ and let $$i_*:\pi_1(\Sigma_g-C_{[\sigma]},x_j)\rightarrow\pi_1(\Sigma_g,x_j)$$ be the surjection induced by the inclusion of spaces $$i:(\Sigma_g-C_{[\sigma]},x_j)\hookrightarrow(\Sigma_g,x_j).$$ We have that 
$$i_*([f_i])= \begin{cases} [f_i] &\mbox{if } 1\leq i\leq 2g \\ 
1 & \mbox{if }  2g+1\leq i\leq 2g+\# C_{[\sigma]}-1 . \end{cases}$$
Since $[\sigma]\in\mathrm{ker}(\omega)$, it is clear that $i_*([\sigma])\in\mathrm{ker}(\Phi)$, where $$\Phi:\pi_1(\Sigma_g,x_j)\rightarrow H_1(\Sigma_g;\mathbb{Z})$$ denotes the Hurewicz map. By the Hurewicz theorem, $i_*([\sigma])$ is in the commutator subgroup $\big[\pi_1(\Sigma_g,x_j),\pi_1(\Sigma_g,x_j)\big]$ of $\pi_1(\Sigma_g,x_j)$. Using Lemma \ref{comm} and Remark \ref{qebbb}, it follows that $i_*([\sigma])$ is a quasi-edge braid. By writing $[\sigma]$ as a unique word on the generators, this means that the product of all generators $[f_i]$ with $1\leq i\leq 2g$ in the word (in order of appearance) is a quasi-edge braid. Let $S=\{[f_1], \ldots, [f_{2g}]\}$. Then by Proposition \ref{represent}, we may assume that $[\sigma]$ is expressed by an $S$-connected word $w$ such that all such elements in $w$ are either commutators or $[f_{i}]$ for some $2g+1\leq i\leq 2g+\#C_{[\sigma]}-1$. By the $S$-connectivity of $w$, it suffices to show that the product of all elements in $w$ that are \textit{not} in $S$ is a quasi-edge braid. This holds by Lemma \ref{split} and Remark \ref{qebbb}, so it follows that $[\sigma]$ is a quasi-edge braid. This completes the proof.

\end{proof}

\section{Future Work}
In this paper, we constructed a natural map $\omega:B_n(\Sigma_g)\rightarrow H_1(\Sigma_g; \mathbb{Z})$ and studied its kernel using simplicial triangulations of $\Sigma_g$. In particular, we showed that $\mathrm{ker}(\omega)$ is generated by canonical braids which are constructed using edges in the triangulation. There are several avenues for future investigation.\\
\begin{itemize}
\item Since $\mathrm{ker}(\omega)$ is generated by edge braids, it would be useful to find a minimal set of relations between the edge braids so that one can build a presentation of $\mathrm{ker}(\omega)$. Lemma \ref{local} gives evidence that such a group presentation might be
similar to Artin's \cite{name} presentation of the classical braid groups $B_n$.\\
\item One can also consider braid groups on non-orientable surfaces (See \cite{non}). Furthermore, the homomorphism $\omega$ can be constructed in exactly the same way for general topological spaces. Thus, it is natural to ask whether there is an analog of Theorem \ref{theorem} for more general surfaces (e.g non-orientable surfaces) than the closed orientable surfaces $\Sigma_g$. The proof of such a result would probably be similar, except it would depend on properties of the fundamental groups of non-orientable surfaces.

\end{itemize}

\end{document}